\documentclass[11pt]{amsart}

\usepackage{amsfonts, fancyhdr, bbm, qtree, amsmath, amscd, tipa, amssymb, hyperref, url, ifsym, amsthm, subfigure, enumitem, wasysym, tikz}
\usetikzlibrary{matrix, arrows}

\newtheorem{thm}{Theorem}[section]
\newtheorem{lem}[thm]{Lemma}

\newtheorem{prop}[thm]{Proposition}

\theoremstyle{definition}
\newtheorem{df}[thm]{Definition}
\theoremstyle{remark}
\newtheorem{rem}[thm]{Remark}

\numberwithin{equation}{section}

\title{On Order Ideals of Minuscule Posets III: The CDE Property}
\author{David B Rush}
\address{Department of Mathematics, Massachusetts Institute of Technology, 77 Massachusetts Ave, Cambridge, MA 02139}
\email{dbr@mit.edu}
\date{\today}

\begin{document}

\begin{abstract}
Recent work of Hopkins establishes that the lattice of order ideals of a minuscule poset satisfies the coincidental down-degree expectations property of Reiner, Tenner, and Yong.  His approach appeals to the classification of minuscule posets.  A uniform proof is presented herein.  

The blueprint follows that of Rush and Wang in their uniform proof that various cardinality statistics are homomesic on orbits of order ideals of minuscule posets under the Fon-Der-Flaass action.  The underpinning remains the original insight of Rush and Shi into the structure of the isomorphism between the weight lattice of a minuscule representation of a complex simple Lie algebra and the lattice of order ideals of the corresponding minuscule poset.  
\end{abstract}

\maketitle

\section{Introduction}

Combinatorial phenomena first observed in products of two chains are often subsequently seen to be present throughout a broader class of partially ordered sets arising from the representation theory of Lie algebras.  Since 2011, the author, in collaboration with his co-authors, has endeavored to explain the ubiquity of these \textit{minuscule posets} by providing uniform proofs of several properties that hold for all minuscule posets.  

In the first of this series \cite{Rush}, the author, together with Shi, proved uniformly that the Fon-Der-Flaass action on order ideals of a minuscule poset obeys the \textit{cyclic sieving phenomenon} of Reiner, Stanton, and White \cite{Reiner}.  This generalized prior work of Fon-Der-Flaass \cite{FDF}, Stanley \cite{StanleyP}, and Striker--Wiliams \cite{Striker}, who considered minuscule posets arising from Lie algebras of types A and B.  

In the sequel \cite{Rush2}, the author, together with Wang, revisited the Fon-Der-Flaass action on order ideals of a minuscule poset and proved uniformly that the order ideal cardinality and antichain cardinality statistics exhibit \textit{homomesy}, as defined by Propp and Roby \cite{Propp}.  This generalized prior work of Propp--Roby \cite{Propp}, who considered minuscule posets arising from Lie algebras of type A, which are products of two chains.  

This continuation adapts the techniques of \cite{Rush} and \cite{Rush2} to prove uniformly that the lattice of order ideals of a minuscule poset satisfies the \textit{coincidental down-degree expectations (CDE)} property, recently introduced by Reiner, Tenner, and Yong \cite{Reiner2}.  The special case of a product of two chains served as the primary combinatorial result of Chan, Mart\'in, Pflueger, and i Bigas \cite{Chan} (who obtained it en route to an application in enumerative geometry).  That result --- along with its generalization by Chan, Haddadan, Hopkins, and Moci \cite{Chan2} --- motivated the Reiner--Tenner--Yong \cite{Reiner2} definition.  

Reiner, Tenner, and Yong \cite{Reiner2} also introduced a stronger incarnation of the CDE property, which they called \textit{multichain-CDE (mCDE)}.  They demonstrated case-by-case that CDE extends to the lattice of order ideals of all minuscule posets, and they conjectured the same for mCDE.  Hopkins \cite{Hopkins} resolved their conjecture with the introduction of the \textit{toggle-CDE (tCDE)} property, stronger yet than mCDE, and a case-by-case proof for tCDE.  We offer a uniform proof of Hopkins's result that connects the tCDE property directly to the representation theory underlying the minuscule posets.  This constitutes a third consummation of the approach inaugurated in Rush--Shi \cite{Rush} and augmented in Rush--Wang \cite{Rush2}, further testifying to our framework's ability to unearth algebraic explanations for (ostensible) combinatorial coincidences.  

\subsection{Toggling and the CDE properties}

Let $Q$ be a poset.  The \textit{down-degree} of an element $p \in Q$ is the number of elements in $Q$ covered by $p$.  The CDE properties concern the expected value of the down-degree function $\operatorname{ddeg} \colon Q \rightarrow \mathbb{R}$ with respect to various probability distributions on $Q$.  

Following Hopkins \cite{Hopkins}, we let $\operatorname{uni}$ denote the uniform distribution on $Q$.  Let $\mathcal{C}_{\max}$ denote the set of maximal chains of $Q$.  Let $\operatorname{maxchain}$ denote the distribution given by \[\mathbb{P}(p) := \frac{| \lbrace c \in \mathcal{C}_{\max} : p \in c \rbrace |}{| \lbrace (c, q) \in \mathcal{C}_{\max} \times Q : q \in c \rbrace |}.\]

\begin{df}
A poset $Q$ satisfies the \textit{CDE} property if $\mathbb{E}(\operatorname{uni}; \operatorname{ddeg}) = \mathbb{E}(\operatorname{maxchain}; \operatorname{ddeg})$.  
\end{df}

For nonnegative integers $k$, let $\mathcal{C}_k$ denote the set of $k$-chains of $Q$, and let $\operatorname{chain}(k)$ denote the distribution given by \[\mathbb{P}(p) := \frac{| \lbrace c \in \mathcal{C}_k : p \in c \rbrace |}{| \lbrace (c, q) \in \mathcal{C}_k \times Q : q \in c \rbrace |}.\]

If $Q$ is a ranked poset with maximal rank $r$, then these distributions interpolate between the distributions defined above, for $\operatorname{uni} = \operatorname{chain}(0)$ and $\operatorname{maxchain} = \operatorname{chain}(r)$.  This observation leads us to the stronger notion of mCDE.  

\begin{df}
A ranked poset $Q$ with maximal rank $r$ satisfies the \textit{mCDE} property if $\mathbb{E}(\operatorname{uni}; \operatorname{ddeg}) = \mathbb{E}(\operatorname{chain(k)}; \operatorname{ddeg})$ for all $1 \leq k \leq r$.  
\end{df}

Suppose $Q$ is a distributive lattice, and let $P$ be the poset of join-irreducible elements of $Q$.  Then $Q$ coincides with $J(P)$ --- the lattice of order ideals of $P$, partially ordered by inclusion.  In this case, Chan, Haddadan, Hopkins, and Modi \cite{Chan2} showed that the distributions $\operatorname{chain(k)}$ on $Q$ exhibit \textit{toggle symmetry}, which they defined to mean that each element in $P$ is equally likely to be eligible to be toggled in as eligible to be toggled out.  

Toggles, so named by Striker and Williams \cite{Striker}, are local actions on order ideals of posets.  For $p \in P$ and $I \in J(P)$, \textit{toggling $I$ at $p$} yields the symmetric difference $I \triangle \lbrace p \rbrace$ if $I \triangle \lbrace p \rbrace \in J(P)$ and returns $I$ otherwise.  We write $t_p \colon J(P) \rightarrow J(P)$ for the action of toggling at $p$, and we consider $p$ eligible to be toggled into $I$ if $t_p(I) = I \cup \lbrace p \rbrace$ and eligible to be toggled out of $I$ if $t_p(I) = I \setminus \lbrace p \rbrace$.  To capture eligibility formally, we define the following indicators: 
\begin{align*}
& \mathcal{T}_p^+ := \mathbbm{1}_{t_p(I) = I \cup \lbrace p \rbrace} \\
& \mathcal{T}_p^- := \mathbbm{1}_{t_p(I) = I \setminus \lbrace p \rbrace} \\
& \mathcal{T}_p := \mathcal{T}_p^+ - \mathcal{T}_p^-
\end{align*}

\begin{df}
A probability distribution $\mu$ on a distributive lattice $J(P)$ is \textit{toggle-symmetric} if $\mathbb{E}(\mu; \mathcal{T}_p^+) = \mathbb{E}(\mu; \mathcal{T}_p^-)$ for all $p \in P$.  
\end{df}

Hopkins's tCDE property requires agreement in down-degree expectation among all toggle-symmetric distributions.  

\begin{df}
A distributive lattice $J(P)$ satisfies the \textit{tCDE} property if $\mathbb{E}(\operatorname{uni}; \operatorname{ddeg}) = \mathbb{E}(\mu; \operatorname{ddeg})$ for all toggle-symmetric probability distributions $\mu$ on $J(P)$.  
\end{df}

We now state our main result before delving into the representation theory from which we obtain our uniform proof.  

\begin{thm} \label{prel}
Let $P$ be a minuscule poset.  Then $J(P)$ satisfies the tCDE property.  
\end{thm}

\subsection{Minuscule Posets} 

Let $\mathfrak{g}$ be a complex simple Lie algebra with Weyl group $W$ and weight lattice $\Lambda$.  Let $\lambda \in \Lambda$ be dominant, and let $V^{\lambda}$ be the irreducible $\mathfrak{g}$-representation with highest weight $\lambda$.  If $W$ acts transitively on the weights of $V^{\lambda}$, we say that $V^{\lambda}$ is a \textit{minuscule $\mathfrak{g}$-representation} with \textit{minuscule weight} $\lambda$.  In this case, the restriction to $W \lambda$ of the partial order on $\Lambda$ opposite to the root order is a distributive lattice, and the poset of join-irreducible elements $P_{\lambda}$ is the \textit{minuscule poset} for $V^{\lambda}$.  Our first task is to construct a map that captures the relationship $J(P_{\lambda}) \cong W \lambda$.  

Note that if $I \lessdot I'$ is a covering relation in $J(P_{\lambda})$, then there exists an element $p \in P_{\lambda}$ such that $I' \setminus I = \lbrace p \rbrace$, and toggling at $p$ transitions back and forth between $I$ and $I'$.  Similarly, if $\mu \lessdot \mu'$ is a covering relation in $W \lambda$, then there exists a simple root $\alpha \in \Lambda$ such that $\mu - \mu' = \alpha$, and the simple reflection $s_{\alpha} \in W$ interchanges $\mu$ and $\mu'$.  

The \textit{minuscule heap} for $V^{\lambda}$ appends to $P_{\lambda}$ the assignment of a simple root $\alpha(p)$ of $\mathfrak{g}$ to each element $p \in P_{\lambda}$ in order to identify each covering relation in $J(P_{\lambda})$ with a counterpart in $W \lambda$.  By stipulating a simple reflection $s_{\alpha(p)}$ to correspond to toggling at $p$ for all $p \in P_{\lambda}$, we realize an explicit isomorphism $J(P_{\lambda}) \xrightarrow{\sim} W \lambda$.  

Let $I$ be an order ideal of $P_{\lambda}$, and let $(p_1, p_2, \ldots, p_{\ell})$ be a linear extension of $I$.  Then $I = t_{p_{\ell}} t_{p_{\ell-1}} \cdots t_{p_1} (\varnothing)$, and we define \[\phi(I) := s_{\alpha(p_{\ell})} s_{\alpha(p_{\ell -1 })} \cdots s_{\alpha(p_1)} (\lambda).\]  

That $\phi \colon J(P_{\lambda}) \rightarrow W \lambda$ is a well-defined isomorphism is due to Stembridge \cite{Stembridge}, who introduced the heap labeling.  In Rush--Shi \cite{Rush}, it is shown that the the action of each simple reflection $s_{\alpha}$ actually corresponds under $\phi$ to that of the sequence of toggles at all elements of $P_{\lambda}$ to which the label $\alpha$ is affixed.  This lemma is the key to our work on minuscule posets.  

\begin{lem}[Rush--Shi \cite{Rush}] \label{equiviso}
Let $V$ be a minuscule $\mathfrak{g}$-representation with minuscule weight $\lambda$, and let $P_{\lambda}$ be the minuscule heap for $V$.  Let $\alpha$ be a simple root of $\mathfrak{g}$.  Let $t_{\alpha} := \prod_{p \in P_{\lambda}^{\alpha}} t_p$.  Then the following diagram is commutative.  
\[\renewcommand{\arraystretch}{1.0}
\begin{array}[c]{ccc}
J(P_{\lambda}) & \stackrel{\phi}{\rightarrow} & W \lambda \\
\downarrow \scriptstyle{t_{\alpha}} && \downarrow \scriptstyle{s_{\alpha}} \\
J(P_{\lambda}) & \stackrel{\phi}{\rightarrow} & W \lambda
\end{array}
\]
\end{lem}

We reframe our main result as follows.  

\begin{thm} \label{main}
Let $V$ be a minuscule $\mathfrak{g}$-representation with minuscule weight $\lambda$ and minuscule heap $P_{\lambda}$.  Suppose $\mathfrak{g}$ is simply laced, and let $\Omega$ denote the common length of the roots.  Then $P_{\lambda}$ satisfies the tCDE property, and, for all toggle-symmetric distributions $\mu$ on $J(P_{\lambda})$, 
\begin{align}
\mathbb{E}(\mu; \operatorname{ddeg}) = 2 \frac{(\lambda, \lambda)}{\Omega^2}.
\end{align}
\end{thm}

\begin{rem}
Given a minuscule poset $P$, there exists a simply laced $\mathfrak{g}$ for which $P$ arises as the minuscule poset of a minuscule $\mathfrak{g}$-representation.  Hence Theorem~\ref{main} indeed entails Theorem~\ref{prel}.
\end{rem}

Our approach to the proof of Theorem~\ref{main} follows the Rush--Wang \cite{Rush2} paradigm: express everything possible in terms of inner products of roots and weights.  A principal innovation of Rush--Wang \cite{Rush2} in this direction is a formula for the number of elements in an order ideal $I$ of a minuscule heap $P_{\lambda}$ labeled by a simple root $\alpha$, which they denote by $f^{\alpha}(I)$.  

For $p \in P_{\lambda}$ and $I \in J(P_{\lambda})$, we find formulas that relate linear combinations of indicator values $\mathcal{T}_p^+(I)$ and $\mathcal{T}_p^-(I)$ to inner products and the cardinalities $f^{\alpha}(I)$.  Incorporating the Rush--Wang \cite{Rush2} formula into our formulas, we are ultimately able to show $\operatorname{ddeg}$ differs by a constant from a particular combination of indicators, for which the expectation with respect to any toggle-symmetric distribution is zero.  

\subsection{Applications}

We close the introduction with the observation that our proof of Theorem~\ref{main} immediately results in a uniform proof that the antichain cardinality statistic exhibits homomesy with respect to the Fon-Der-Flaass action and gyration on order ideals of a minuscule poset.  To explain, we recall the definition of homomesy.  

\begin{df} \label{homomesy}
Let $\mathcal{S}$ be a finite set, and let $\tau \colon \mathcal{S} \rightarrow \mathcal{S}$ be an action on $\mathcal{S}$ of order $m$.  A function $f \colon \mathcal{S} \rightarrow \mathbb{R}$ \textit{exhibits homomesy} with respect to $\tau$ if there exists a constant $c \in \mathbb{R}$ such that, for all $x \in S$, the following equality holds:
\[
\frac{1}{m} \sum_{i=0}^{m-1} f(\tau^i(x)) = c.
\]

In this case, $f$ is \textit{$c$-mesic} with respect to $\tau$. 
\end{df}

The connection between tCDE and homomesy is spelled out in Hopkins \cite{Hopkins}.  In brief, let $P$ be a poset, and let $\Phi$ be an action on $J(P)$.  The \textit{antichain cardinality statistic} on $J(P)$ is just the function $\operatorname{ddeg} \colon J(P) \rightarrow \mathbb{R}$.  Suppose that, for each $\Phi$-orbit $\mathcal{O} \subset J(P)$, the distribution supported uniformly on $\mathcal{O}$ is toggle-symmetric.  If $J(P)$ satisfies the tCDE property, then $\operatorname{ddeg}$ is $c$-mesic with respect to $\Phi$ with $c = \mathbb{E}(\operatorname{uni}; \operatorname{ddeg})$.  

For ranked posets $P$ (including all minuscule posets), a result of Striker \cite{Strikertoggle}, as phrased by Chan, Haddadan, Hopkins, and Moci \cite{Chan2}, is that the distribution supported uniformly on a $\Phi$-orbit $\mathcal{O} \subset J(P)$ is toggle-symmetric if $\Phi$ is the Fon-Der-Flaass action or gyration.  Hence the antichain cardinality statistic is $c$-mesic with respect to both actions with $c = 2 \frac{(\lambda, \lambda)}{\Omega^2}$.

Thus, for the Fon-Der-Flaass action, we obtain an alternate uniform proof of Theorem 1.4 of Rush--Wang \cite{Rush2}, and we achieve a new result (and new uniform proof) for gyration.  

The rest of this article is organized as follows.  In section 2, we review the background on minuscule posets and minuscule heaps, and we restate the Rush--Shi \cite{Rush} lemma.  In section 3, we state the Rush--Wang \cite{Rush2} formula and derive our formulas.  Then we show how Theorem~\ref{main} follows.  

\section{Minuscule Posets and Minuscule Heaps}

This section is an overview of the background behind the minuscule posets and their labeled incarnations, the minuscule heaps.  Because we rely on the same preliminaries as Rush and Wang \cite{Rush2}, we content ourselves with a summary of section 2 of \cite{Rush2}, to which we refer the curious reader for a more thorough treatment of the requisite material.  

\subsection{Minuscule Posets}

Let $\mathfrak{g}$ be a complex simple Lie algebra, and let $\mathfrak{h}$ be a choice of Cartan subalgebra.  Let $R \subset \mathfrak{h}^*$ be the roots of $\mathfrak{g}$, and let $\mathfrak{h}_{\mathbb{R}}^*$ be the real vector space generated by $R$.  Recall that $\mathfrak{h}_{\mathbb{R}}^*$ is equipped with an inner product, which we denote by $(\cdot, \cdot)$, induced from the Killing form on $\mathfrak{g}$ via its restriction to $\mathfrak{h}$.  Furthermore, $R$ is a reduced root system for the inner product space $\mathfrak{h}_{\mathbb{R}}^*$  (cf. Kirillov \cite{Kirillov}, Theorem 7.3). 

For all $\alpha \in R$, let $\alpha^{\vee} := 2 \frac{\alpha}{(\alpha,\alpha)}$ be the coroot associated to $\alpha$.  Let $\Pi = \lbrace \alpha_1, \alpha_2, \ldots, \alpha_t \rbrace$ be a choice of set of simple roots for $R$, and write $\Pi^{\vee} := \lbrace \alpha_1^{\vee}, \alpha_2^{\vee}, \ldots, \alpha_t^{\vee} \rbrace$ for the corresponding set of simple coroots.  The sets $\Pi$ and $\Pi^{\vee}$ both constitute bases for $\mathfrak{h}_{\mathbb{R}}^*$, and we refer to the lattices $\Phi$ and $\Phi^{\vee}$ in $\mathfrak{h}_{\mathbb{R}}^*$ generated over $\mathbb{Z}$ by $\Pi$ and $\Pi^{\vee}$ as the \textit{root} and \textit{coroot lattices}, respectively.  

The \textit{weight lattice} $\Lambda$, which comprises the weights of $\mathfrak{g}$, is the dual lattice to the coroot lattice $\Phi^{\vee}$.  The basis corresponding to $\Pi^{\vee}$ is the set of \textit{fundamental weights}.  

\begin{df} \label{weight}
A functional $\lambda \in \mathfrak{h}_{\mathbb{R}}^*$ is a \textit{weight} of $\mathfrak{g}$ if $(\lambda, \alpha_i^{\vee}) \in \mathbb{Z}$ for all $1 \leq i \leq t$.  
\end{df}

\begin{df} \label{dominant}
A weight $\lambda \in \Lambda$ is \textit{dominant} if $(\lambda, \alpha_i^{\vee}) \geq 0$ for all $1 \leq i \leq t$.  
\end{df}  

\begin{df} \label{fundamental}
The \textit{fundamental weights} $\omega_1, \omega_2, \ldots, \omega_t$ are defined by the relations $(\omega_i, \alpha_j^{\vee}) = \delta_{ij}$ for $1 \leq i, j \leq t$, where $\delta_{ij}$ denotes the Kronecker delta.     
\end{df}

The dominant weights in $\Lambda$ index the finite-dimensional irreducible representations of $\mathfrak{g}$.  

\begin{thm}[Kirillov \cite{Kirillov}, Corollary 8.24] \label{weightrep}
For all $\lambda \in \Lambda^+$, there exists a finite-dimensional irreducible representation $V^{\lambda}$ of $\mathfrak{g}$ with highest weight $\lambda$.  Furthermore, the map $\lambda \mapsto [V^{\lambda}]$ defines a bijection between $\Lambda^+$ and the set of isomorphism classes of finite-dimensional irreducible $\mathfrak{g}$-representations.    
\end{thm}

Let $V$ be a finite-dimensional irreducible representation of $\mathfrak{g}$.  For all $\mu \in \Lambda$, the weight $\mu$ of $\mathfrak{g}$ is a \textit{weight} of $V$ if the weight space associated to $\mu$, namely, \[\lbrace v \in V : hv = \mu(h)v \text{  } \forall h \in \mathfrak{h} \rbrace,\] is nonzero.  For all $\lambda \in \Lambda^+$, we write $\Lambda_{\lambda} \subset \Lambda$ for the (finite) subset comprising the weights of $V^{\lambda}$.  

For all roots $\alpha \in R$, let the \textit{reflection} $s_{\alpha}$ associated to $\alpha$ be the orthogonal involution on $\mathfrak{h}_{\mathbb{R}}^*$ given by $\lambda \mapsto \lambda - (\lambda, \alpha^{\vee}) \alpha$.  Recall that the \textit{Weyl group} $W$ of $\mathfrak{g}$ is the subgroup of $O(\mathfrak{h}_{\mathbb{R}}^*)$ generated by the set of \textit{simple reflections} $\lbrace s_{i}\rbrace_{i=1}^t$, where $s_i := s_{\alpha_i}$ for all $1 \leq i \leq t$.  For all $\lambda \in \Lambda^+$, the action of $W$ on $\Lambda$ restricts to an action on $\Lambda_{\lambda}$ (cf. Kirillov \cite{Kirillov}, Theorem 8.8).  

At last we come to the definition of a minuscule representation, which underlies that of a minuscule poset.  

\begin{df} \label{minurep}
Let $\lambda$ be a dominant weight of $\mathfrak{g}$.  Then $V^{\lambda}$ is \textit{minuscule} with \textit{minuscule weight} $\lambda$ if the action of $W$ on $\Lambda_{\lambda}$ is transitive.  
\end{df}

The last ingredient in the definition of a minuscule poset is a partial order on the weights $\Lambda_{\lambda}$.  One choice, and that taken in Rush--Shi \cite{Rush}, is the restriction to $\Lambda_{\lambda}$ of the \textit{root order} on $\Lambda$, viz., the transitive closure of the relations $\mu \lessdot \nu$ for $\mu, \nu \in \Lambda$ satisfying $\nu - \mu \in \Pi$.  Here, however, we follow Proctor \cite{Proctor} and Rush--Wang \cite{Rush2}, and opt for the opposite order on $\Lambda_{\lambda}$.  

\begin{prop}
Let $\lambda \in \Lambda^+$ such that the $\mathfrak{g}$-representation $V^{\lambda}$ is minuscule.  Then $\Lambda_{\lambda}$ is a distributive lattice.  
\end{prop}

\begin{proof}
See Proctor \cite{Proctor}, Propositions 3.2 and 4.1.  For a uniform proof, see Stembridge \cite{Stembridge}, Theorems 6.1 and 7.1.  
\end{proof}

\begin{df}
Let $V$ be a minuscule representation of $\mathfrak{g}$ with minuscule weight $\lambda$.  The restriction of the partial order on $\Lambda_{\lambda}$ to its join-irreducible elements is the \textit{minuscule poset} for $V$, which we denote by $P_{\lambda}$.  
\end{df}

\begin{df}
Let $P$ be a poset.  Then $P$ is \textit{minuscule} if there exists a complex simple Lie algebra $\mathfrak{g}$ and a dominant weight $\lambda$ of $\mathfrak{g}$ for which the $\mathfrak{g}$-representation $V^{\lambda}$ is minuscule and $P \cong P_{\lambda}$.  
\end{df}

\subsection{Minuscule Heaps}

Suppose that $\lambda$ is a dominant weight of $\mathfrak{g}$ for which $V^{\lambda}$ is minuscule.  For all $\mu \in \Lambda_{\lambda}$, we denote by $\Lambda_{\lambda}^{\mu}$ the restriction of the partial order on $\Lambda_{\lambda}$ to the set $\lbrace \nu \in \Lambda_{\lambda} : \nu \leq \mu \rbrace$.  To realize an explicit isomorphism $J(P_{\lambda}) \cong \Lambda_{\lambda}$, we consider a family of labeled posets $\lbrace P_{\lambda, \mu} \rbrace_{\mu \in \Lambda_{\lambda}}$, which we refer to as \textit{heaps}.  

What follows is a brief review of the theory of heaps in the context of minuscule representations.  Taken together, the results of Proctor \cite{Proctor}, Stembridge \cite{Stembridge}, and Stembridge \cite{Stembridge2} suffice to prove any claim in 2.10 -  2.18.  Stembridge's proof of a slightly more general version of Theorem~\ref{structure} is paraphrased in Rush--Shi \cite{Rush}, but a cleaner proof of Theorem~\ref{structure} itself, specific to the minuscule setting, is carried out in Rush--Wang \cite{Rush2} via Proposition~\ref{heapiso}.

Theorem~\ref{structure} builds the explicit isomorphism $J(P_{\lambda}) \cong \Lambda_{\lambda}$.  Then Theorem~\ref{rushshi} presents the insight of Rush--Shi \cite{Rush} into its inner workings.

\begin{prop}[Bourbaki \cite{Bourbaki}, Exercise VI.I.24] \label{bourbaki}   
Let $\mu \in \Lambda_{\lambda}$.  Then $(\mu, \alpha_i^{\vee}) \in \lbrace -1, 0, 1 \rbrace$ for all $1 \leq i \leq t$.  
\end{prop}

\begin{prop} \label{cover}
Let $\mu \in \Lambda_{\lambda}$.  If $\mu \lessdot \mu - \alpha_i$ is a covering relation, then $s_i (\mu) = \mu - \alpha_i$.
\end{prop}

\begin{df} \label{lamminus}
Let $w \in W$.  Then $w$ is \textit{$\lambda$-minuscule} if there exists a reduced word $w = s_{i_{\ell}} s_{i_{\ell-1}} \cdots s_{i_1}$ such that \[\lambda \lessdot s_{i_1} \lambda \lessdot \cdots \lessdot (s_{i_{\ell}} s_{i_{\ell-1}} \cdots s_{i_1}) \lambda = w \lambda\] is a saturated chain in $\Lambda_{\lambda}$.  
\end{df}

\begin{prop} \label{everyred}
Let $w \in W$.  If $w$ is $\lambda$-minuscule and $w = s_{i_{\ell}} s_{i_{\ell-1}} \cdots s_{i_1}$ is a reduced word, then \[\lambda \lessdot s_{i_1} \lambda \lessdot \cdots \lessdot (s_{i_{\ell}} s_{i_{\ell-1}} \cdots s_{i_1}) \lambda = w \lambda\] is a saturated chain in $\Lambda_{\lambda}$.  
\end{prop}

\begin{prop} \label{everysat}
Let $\mu \in \Lambda_{\lambda}$.  Then there exists a unique $\lambda$-minuscule element $w \in W$ such that $w \lambda = \mu$.  Furthermore, if \[\lambda \lessdot \lambda - \alpha_{i_1} \lessdot \cdots \lessdot \lambda - \alpha_{i_1} - \alpha_{i_2} - \cdots - \alpha_{i_{\ell}} = \mu\] is a saturated chain in $\Lambda_{\lambda}$, then $s_{i_{\ell}} s_{i_{\ell-1}} \cdots s_{i_1}$ is a reduced word for $w$.  
\end{prop}

\begin{df} \label{heap}
Let $w \in W$ be $\lambda$-minuscule, and let $w = s_{i_{\ell}} s_{i_{\ell-1}} \cdots s_{i_1}$ be a reduced word.  The \textit{heap} $P_{\lambda, (i_1, i_2, \ldots, i_{\ell})}$ associated to $s_{i_{\ell}} s_{i_{\ell -1}} \cdots s_{i_1}$ is the labeled set $\lbrace 1, 2, \ldots, \ell \rbrace$, where $i_j$ is the \textit{label} of the element $j$ for all $1 \leq j \leq \ell$, equipped with the partial order arising as the transitive closure of the relations $j < j'$ for all $1 \leq j < j' \leq \ell$ for which $s_{i_j}$ and $s_{i_{j'}}$ do not commute.  
\end{df}

\begin{prop} \label{linext}
Let $w \in W$ be $\lambda$-minuscule, and let $w = s_{i_{\ell}} s_{i_{\ell-1}} \cdots s_{i_1}$ be a reduced word.  Let $\mathcal{L}\left(P_{\lambda, (i_1, i_2, \ldots, i_{\ell})}\right) := \lbrace \mathcal{A} : \mathcal{A}_1 \lessdot \mathcal{A}_2 \lessdot \cdots \lessdot \mathcal{A}_{\ell} \rbrace$ be the set of linear extensions of the heap $P_{\lambda, (i_1, i_2, \ldots, i_{\ell})}$.  For all $\mathcal{A} \in \mathcal{L}$, let $s(A) := s_{i_{\mathcal{A}_{\ell}}} s_{i_{\mathcal{A}_{\ell-1}}} \cdots s_{i_{\mathcal{A}_1}}$.  Then $\lbrace s(A) \rbrace_{\mathcal{A} \in \mathcal{L}\left(P_{\lambda, (i_1, i_2, \ldots, i_{\ell})}\right)}$ is the set of reduced words for $w$ in $W$.  
\end{prop}

\begin{prop} \label{totorder}
Let $w \in W$ be $\lambda$-minuscule, and let $w = s_{i_{\ell}} s_{i_{\ell-1}} \cdots s_{i_1}$ be a reduced word.  In the heap $P_{\lambda, (i_1, i_2, \ldots, i_{\ell})}$, if $i_j = i_k$, then $j < k$ or $k < j$.  
\end{prop}

\begin{prop} \label{fullcomm}
Let $w \in W$ be $\lambda$-minuscule, and let $s_{i_{\ell}} s_{i_{\ell-1}} \cdots s_{i_1}$ and $s_{i'_{\ell}} s_{i'_{\ell-1}} \cdots s_{i'_1}$ be two reduced words for $w$.  Then there exists a sequence of commuting braid relations (viz., relations of the form $s_p s_q = s_q s_p$ for commuting simple reflections $s_p, s_q$) exchanging $s_{i_{\ell}} s_{i_{\ell-1}} \cdots s_{i_1}$ and $s_{i'_{\ell}} s_{i'_{\ell-1}} \cdots s_{i'_1}$.  
\end{prop}

\begin{prop} \label{heapiso}
Let $w \in W$ be $\lambda$-minuscule, and let $s_{i_{\ell}} s_{i_{\ell-1}} \cdots s_{i_1}$ and $s_{i'_{\ell}} s_{i'_{\ell-1}} \cdots s_{i'_1}$ be two reduced words for $w$.  Then there exists a unique permutation $\sigma \in \mathfrak{S}_{\ell}$ such that $\sigma \colon \lbrace 1, 2, \ldots, \ell \rbrace \rightarrow \lbrace 1, 2, \ldots, \ell \rbrace$ defines an isomorphism of heaps $P_{\lambda, (i_1, i_2, \ldots, i_{\ell})} \xrightarrow{\sim} P_{\lambda, (i'_1, i'_2, \ldots, i'_{\ell})}$.  
\end{prop}

\begin{thm} \label{structure}
Let $\mu \in \Lambda_{\lambda}$.  Given an order ideal $I \in J(P_{\lambda, \mu})$, let $\mathcal{A}_I$ be a linear extension of $I$, and set $\phi(I) := s(\mathcal{A}_I) \lambda$.  Then $\phi$ defines an isomorphism of posets $J(P_{\lambda, \mu}) \xrightarrow{\sim} \Lambda_{\lambda}^{\mu}$.  
\end{thm}

\begin{df}
Let $V$ be a minuscule representation of $\mathfrak{g}$ with minuscule weight $\lambda$.  The heap $P_{\lambda, w_0 \lambda}$, which we denote by $P_{\lambda}$, is the \textit{minuscule heap} for $V$.  
\end{df}

\begin{thm}[Rush--Shi \cite{Rush}, Theorem 6.3] \label{rushshi}
Let $V$ be a minuscule representation of $\mathfrak{g}$ with minuscule weight $\lambda$, and let $P_{\lambda}$ be the minuscule heap for $V$.  For all $1 \leq i \leq t$, let $P_{\lambda}^i \subset P_{\lambda}$ be the set of elements of $P_{\lambda}$ labeled by $i$, and let $t_i := \prod_{p \in P_{\lambda}^i} t_p$.  Then the following diagram is commutative.  
\[\renewcommand{\arraystretch}{1.0}
\begin{array}[c]{ccc}
J(P_{\lambda}) & \stackrel{\phi}{\rightarrow} & \Lambda_{\lambda} \\
\downarrow \scriptstyle{t_i} && \downarrow \scriptstyle{s_i} \\
J(P_{\lambda}) & \stackrel{\phi}{\rightarrow} & \Lambda_{\lambda}
\end{array}
\]
\end{thm}

\begin{rem}
From Proposition~\ref{totorder}, we see that $P_{\lambda}^i$ is a totally ordered subset of $P_{\lambda}$ that contains no covering relations in $P_{\lambda}$.  This observation, which explains why the map $t_i$ is well-defined, is implicit not only to Theorem~\ref{rushshi}, but also to our results in the next section.  
\end{rem}

\section{Proof of Theorem 1.7}
In this section, we prove our main theorem.  We begin with two lemmas.  The first is taken from Rush--Wang \cite{Rush2}.  It expresses in terms of inner products the number of elements with a common label in a heap order ideal.  Then the second lemma does the same for combinations of indicators of toggle eligibility.  

\begin{lem}[Rush--Wang \cite{Rush2}, Lemma 3.4] \label{whatisordercard}
Let $V$ be a minuscule representation of $\mathfrak{g}$ with minuscule weight $\lambda$ and minuscule heap $P_{\lambda}$.  For all $1 \leq i \leq t$, let $P_{\lambda}^i \subset P_{\lambda}$ be the set of elements of $P_{\lambda}$ labeled by $i$, and let $f^i \colon J(P_{\lambda}) \rightarrow \mathbb{R}$ be defined by $I \mapsto |I \cap P_{\lambda}^i|$.  Then \[f^i(I) = 2 \frac{(\lambda, \omega_i) - (\phi(I), \omega_i)}{(\alpha_i, \alpha_i)}.\]  
\end{lem}

\begin{lem} \label{degtog}
Let $V$ be a minuscule representation of $\mathfrak{g}$ with minuscule weight $\lambda$ and minuscule heap $P_{\lambda}$.  Write $P_{\lambda}^i = \lbrace p_{i,1} < p_{i,2} < \cdots < p_{i,k} \rbrace$.  Then the following two equalities hold.  
\begin{align}
& \sum_{j=1}^k \mathcal{T}_{p_{i,j}}(I) =  (\phi(I), \alpha_i^{\vee}). \\
& \sum_{j=1}^k (j-1) \cdot \mathcal{T}_{p_{i,j}}^+(I) - j \cdot \mathcal{T}_{p_{i,j}}^-(I) = f^i(I) (\phi(I), \alpha_i^{\vee})
\end{align}
\end{lem}

\begin{proof}
Suppose that $(\phi(I), \alpha_i^{\vee}) = 1$.  Then $s_i(\phi(I))$ covers $\phi(I)$ in $\Lambda_{\lambda}$, so, by Theorem~\ref{rushshi}, we see that $t_i(I)$ covers $I$ in $J(P_{\lambda})$.  Let $j' = f^i(I) + 1$.  Then toggling $I$ at $p_{i,j'}$ yields $I \cup \lbrace p_{i,j'} \rbrace$, and toggling $I$ at any $p_{i, j''} \in P_{\lambda}^i$ for which $j'' \neq j'$ returns $I$.  Hence $\mathcal{T}_{p_{i,j}}^+(I) = \delta_{j, j'}$, where $\delta_{j, j'}$ denotes the Kronecker delta, and $\mathcal{T}_{p_{i,j}}^-(I) = 0$ for all $1 \leq j \leq k$.    

To complete the proof, we recall from Proposition~\ref{bourbaki} that $(\phi(I), \alpha_i^{\vee}) \in \lbrace -1, 0, 1 \rbrace$, and we note that the cases $(\phi(I), \alpha_i^{\vee}) = 0$ and $(\phi(I), \alpha_i^{\vee})$ are handled by analogous reasoning.  
\end{proof}

Our final lemma combines our first two.  

\begin{lem} \label{final}
Let $V$ be a minuscule representation of $\mathfrak{g}$ with minuscule weight $\lambda$ and minuscule heap $P_{\lambda}$.  Write $P_{\lambda}^i = \lbrace p_{i,1} < p_{i,2} < \cdots < p_{i,k} \rbrace$.  Define the function $X_i \colon J(P_{\lambda}) \rightarrow \mathbb{R}$ by \[X_i := \sum_{j=1}^k \mathcal{T}_{p_{i,j}}^- - \left(\sum_{j=1}^k (j-1) \cdot \mathcal{T}_{p_{i,j}} \right) + 2 \frac{(\lambda, \omega_i)}{(\alpha_i, \alpha_i)} \sum_{j=1}^k \mathcal{T}_{p_{i,j}}.\]
Then \[X_i(I) = \frac{2}{(\alpha_i, \alpha_i)} (\phi(I), \omega_i) (\phi(I), \alpha_i^{\vee}).\]
\end{lem}

\begin{proof}
\begin{align*}
X_i(I) & = \left(\sum_{j=1}^k  j \cdot \mathcal{T}_{p_{i,j}}^-(I) - (j-1) \cdot \mathcal{T}_{p_{i,j}}^+(I) \right) + 2 \frac{(\lambda, \omega_i)}{(\alpha_i, \alpha_i)} \left(\sum_{j=1}^k \mathcal{T}_{p_{i,j}}(I)\right) \\ & = -f^i(I) (\phi(I), \alpha_i^{\vee}) + 2 \frac{(\lambda, \omega_i)}{(\alpha_i, \alpha_i)} (\phi(I), \alpha_i^{\vee}) \\ & = \left(2 \frac{(\phi(I), \omega_i) - (\lambda, \omega_i)}{(\alpha_i, \alpha_i)} + 2 \frac{(\lambda, \omega_i)}{(\alpha_i, \alpha_i)}\right) (\phi(I), \alpha_i^{\vee}) \\ & = \frac{2}{(\alpha_i, \alpha_i)} (\phi(I), \omega_i) (\phi(I), \alpha_i^{\vee}).
\end{align*}
\end{proof}

We proceed to the proof of Theorem~\ref{main}.  Note that \[\operatorname{ddeg} = \sum_{p \in P_{\lambda}} \mathcal{T}_p^- = \sum_{i = 1}^t \sum_{p \in P_{\lambda}^i} \mathcal{T}_p^-.\]  

Recall that if $\mu$ is a toggle-symmetric distribution on $J(P_{\lambda})$, then $\mathbb{E}(\mu; \mathcal{T}_p)$ vanishes for all $p \in P_{\lambda}$.  Thus, for all toggle-symmetric distributions $\mu$ on $J(P_{\lambda})$, we see that $\mathbb{E}(\mu; \operatorname{ddeg}) = \mathbb{E}(\mu; \sum_{i=1}^t X_i)$.  

Now suppose $\mathfrak{g}$ is simply laced, and let $\Omega$ denote the common length of the roots of $\mathfrak{g}$.  Then, for all $I \in J(P_{\lambda})$, 
\begin{align*}
\sum_{i=1}^t X_i(I) & = \sum_{i=1}^t \frac{2}{(\alpha_i, \alpha_i)} (\phi(I), \omega_i) (\phi(I), \alpha_i^{\vee}) \\ & = \frac{2}{\Omega^2} \left(\phi(I), \sum_{i=1}^t (\phi(I), \omega_i) \alpha_i^{\vee} \right) \\ & = \frac{2}{\Omega^2} (\phi(I), \phi(I)) = 2\frac{(\lambda, \lambda)}{\Omega^2},
\end{align*}
where the last equality holds because $W$ is orthogonal and acts transitively on $\Lambda_{\lambda}$.  This completes the proof. \qed

\section{Acknowledgments}

The author's interest in what has become known as the CDE property was prompted by a presentation of Melody Chan at the American Institute of Mathematics workshop on Dynamical Algebraic Combinatorics.  The author extends his gratitude to James Propp, Thomas Roby, Jessica Striker, and Nathan Williams for organizing the workshop and inviting the author to speak about his prior work on minuscule posets.  He also thanks Sam Hopkins, Victor Reiner, and XiaoLin Shi for helpful conversations.  

During the month in which this article was written, the author was a quantitative research intern at the Susquehanna International Group.  He is grateful for SIG's vibrant intellectual environment, which was of considerable benefit to his research.  

The author is presently supported by the US National Science Foundation Graduate Research Fellowship Program.

\end{document}